\documentclass[reqno,11pt]{amsart}

\usepackage{palatino}

\usepackage{amsmath}
\usepackage{amsfonts,amssymb,amsmath,latexsym,wasysym,mathrsfs,bbm,stmaryrd}
\usepackage[all]{xy}
\usepackage[pdftex]{graphicx,color}
\usepackage{enumerate}

\usepackage[mathcal]{euscript}
\usepackage{xcolor}
\usepackage{graphics,moreverb,caption,gastex}

\def\N{\mathbb N}

\def\1{\mathbbm 1}

\newcommand{\interior}[1]{\raise0.2ex\hbox{$\displaystyle{\mathop{#1}^{\circ}}$}}

\newtheorem{theorem}{Theorem}[section]
\newtheorem*{theorem*}{Theorem}

\newtheorem{definition}[theorem]{Definition}

\newtheorem{lemma}[theorem]{Lemma}

\numberwithin{equation}{section}

\theoremstyle{remark}
\newtheorem*{remark*}{Remark}

\theoremstyle{remark}

\long\def\symbolfootnote[#1]#2{\begingroup%
\def\thefootnote{\fnsymbol{footnote}}\footnote[#1]{#2}\endgroup}

\begin{document}

\title{Boundary Traces of Holomorphic Functions on the Unit Ball in $\mathbb{C}^n$}
\author{William Gryc, Muhlenberg College}

\maketitle

\section{Abstract}
It is a classical theorem that if a function is integrable along the boundary of the unit circle, then the function is the nontangential limit of a holomorphic function on the open disc if and only if its Fourier coefficients for nonnegative integers are zero. In this article we generalize this result to higher complex dimensions by proving that for an integrable function on the unit sphere, it is ``boundary trace'' of a holomorphic function on the open unit ball if and only if two particular families of integral equations are satisfied. To do this, we use the theory of Hardy spaces as well as the invariant Poisson and Cauchy integrals.

\section{Introduction}
Holomorphic functions are \textit{special}. This is the major theme of an undergraduate course in complex analysis. Whereas the stars in a real analysis class are the counterexamples and the theorems you thought should be true but are actually false (in full generality at least), the stars in a complex analysis class are facts you would never dream to be true but are, such as the Cauchy Integral formula, Liouville's Theorem, and Hadamard's Three Lines Theorem. Now you can make an argument that lots of families of functions are special, like smooth functions and harmonic functions. But those functions are in many ways ubiquitous and holomorphic functions are much more rare. For example, most integrable functions on a reasonable domain can be approximated by smooth functions, but not by holomorphic functions. If an $L^1$ function saw a smooth function at the supermarket they wouldn't be surprised, but if they saw a holomorphic function they would be tempted to ask it for a selfie. In the same vein, most integrable functions along a boundary of some reasonable subset of $\mathbb{R}^{n}$ are the boundary values of a harmonic function. Only a select few of those integrable functions are the boundary values of a holomorphic function (in the case $n$ is even). This paper is about this latter topic: what does it take for an integrable function on a boundary to be a boundary value of a holomorphic function when the ``reasonable subset'' is the unit ball of $\mathbb{C}^n$? (Admittedly, the unit ball is about as nice of a subset you can find, but it will still lead to an interesting problem.)

Let's make the question above more formal. Fix $n\in \mathbb{N}$. Let $B$ denote the open unit ball and $S$ denote the unit sphere in $\mathbb{C}^n$. That is, for $z=(z_1,\ldots,z_n)\in \mathbb{C}^n$, we have $z\in B$ if and only if $|z_1|^2+\ldots+|z_n|^2<1$ and $z\in S$ if and only if $|z_1|^2+\ldots+|z_n|^2=1$. We say that a function $F:B\to\mathbb{C}$ is \emph{holomorphic} if $F$ is continuous and is holomorphic in each of its $n$ complex variables separately. As we noted above, we are going to look at integrable functions on $S$. This is a large space of functions that will allow us to consider integral formulas for our boundary functions. Thus, we will consider functions in $L^p(S,\sigma)$, where $1\leq p <\infty$, $\sigma$ is the standard surface measure of the sphere $S$ normalized so that $\sigma(S)=1$, and
\[L^p(S,\sigma)=\left\{f:S\to\mathbb{C}\left| \int_{S} |f(\zeta)|^p d\sigma(\zeta) < \infty\right.\right\}.\]
Furthermore, we have the norm $\|f\|_p:=(\int_{S} |f(\zeta)|^p d\sigma(\zeta))^{1/p}$ for $f\in L^p(S,\sigma)$ under which $L^p(S,\sigma)$ is complete.

Next we address what it means for $f$ to be ``the boundary value of a function $F$ that is holomorphic on $B$.'' Let $A(S)$ denote the set of functions on $S$ that are restrictions to $S$ of functions $F$ that are continuous on the closure of $B$ and holomorphic on $B$. While $A(S)$ is naturally a set of ``boundary traces'' of holomorphic functions on $B$ that also lie in $L^p(S,\sigma)$ for all $p\geq 1$, it turns out that we can expand it: for $1\leq p<\infty$ let $H^p(S)$ denote the closure of $A(S)$ in $L^p(S,\sigma)$. While we are mainly interested in the case that $1\leq p<\infty$, $H^p(S)$ can also be defined in this way for $0<p<1$. Theorem \ref{theorem-1} below will show that $H^p(S)$ can be thought of boundary traces of holomorphic functions. Before we can state this theorem, we need to define the Hardy spaces $H^p(B)$ of holomorphic functions on $B$:

\begin{definition}
Let $0<p<\infty$. Define the Hardy space $H^p(B)$ as the set of all holomorphic functions $F$ on $B$ such that 
\[\sup_{0\leq r < 1} \left(\int_{S} |F(r\zeta)|^p d\sigma(\zeta) \right)^{1/p}<\infty.\]
\end{definition}
If $p\geq 1$, the supremum above can be used as a norm on $H^p(B)$ and under this norm $H^p(B)$ is a Banach space (see \cite{Rudin1980}). It turns out that every element of $H^p(B)$ admits a special type of limit for almost every point on the boundary $S$. Kor\'{a}nyi first defined these limits and called them ``admissible limits'' in \cite{Koranyi1969} but we will follow \cite{Rudin1980} and call them $K$-limits. For this paper the detailed definition of the $K$-limit will not be needed, but we note that it is a limit to points on $S$ taken within $B$ but with certain trajectories disallowed; in the case that $n=1$, these limits are usually called \emph{nontangential limits} but this nomenclature is inaccurate for $n>1$ as some tangential trajectories are allowed. For the details of $K$-limits we refer the interested reader to \cite{Koranyi1969} or Chapter 5 in \cite{Rudin1980}. In any case, it turns out that all elements of $H^p(B)$ have $K$-limits that lie in $H^p(S)$ (including the case where $0<p<1$):
\begin{theorem}\label{theorem-1}
Let $0 < p <\infty$. For a function $F$ with domain $B$ and $\zeta\in S$, let $F^*(\zeta)$ denote the $K$-limit of $F$ at $\zeta$, if that limit exists.
\begin{enumerate}[(a)]
\item For any $F\in H^p(B)$, the $K$-limit $F^*$ exists almost everywhere with respect to $\sigma$ on $S$ and $F^*\in H^p(S)$. Furthermore, the mapping $F\mapsto F^*$ is an isometry.
\item If $p\geq 1$ and $f\in H^p(S)$, then $C[f]\in H^p(B)$, $C[f]=P[f]$, and $C[f]^*=f$.
\end{enumerate}
\end{theorem}
For a proof, see \cite{Rudin1980} where the above theorem is stated as Theorem 5.6.8. In Theorem \ref{theorem-1}, we note that $P[f]$ is the (invariant) Poisson integral of $f$ and $C[f]$ is the Cauchy integral of $f$, and we will give detailed definitions of these integrals below. At the moment, though, we note that part (b) of Theorem \ref{theorem-1} implies that the elements of $H^p(S)$ are $K$-limits of holomorphic functions on $B$ and thus $H^p(S)$ is an appropriate choice to serve as our set of boundary traces of holomorphic functions when $p\geq 1$. Indeed, in light of Theorem \ref{theorem-1} we can now restate our original question rigorously:\\

\noindent\textit{Fix $1\leq p<\infty$ and suppose $f\in L^p(S,\sigma)$. We want to find necessary and sufficient integral formulas for $f$ to lie in $H^p(S)$.\\}

Classical Hardy space theory provides an answer in the case that $n=1$. For this case, let $D$ denote the unit disc and $U$ the unit circle in $\mathbb{C}$. We have the following result:

\begin{theorem}\label{theorem-2}
Let $1\leq p<\infty$ and suppose $f\in L^p(U,\sigma)$. Then $f\in H^p(U)$ if and only if for every $j\in\N$ we have
\[\int_{U} \zeta^jf(\zeta)d\sigma(\zeta) = 0.\]
\end{theorem} 
Put another way, $f\in H^p(U)$ if and only if its Fourier coefficients for negative integers are all 0. This result is classical and can be found, for example, as Theorem 3.4 in \cite{Duren1970}. (In fact, more can be said on the other Fourier coefficients: let $F\in H^p(D)$ with $F^*=f$. As $F$ is holomorphic, is has a power series representation centered at $0$: $F(z) = \sum_{j=0}^\infty a_j z^j$. For nonnegative indices $j$, the $j^{th}$ Fourier coefficient of $f$ is equal to $a_j$.) The goal of this paper is to generalize Theorem \ref{theorem-2} for $n>1$.

Before we state our generalization, we state some notation for elements of $\mathbb{C}^n$: first if $x+iy$ is a complex number with $x$ and $y$ real, then $\overline{x+iy}=x-iy$ is its complex conjugate. For $z=(z_1,z_2,\ldots,z_n)\in\mathbb{C}^n$, then $\overline{z}=(\overline{z}_1, \overline{z}_2,\ldots, \overline{z}_n)$. We also need to be able to take ``powers'' of elements of $\mathbb{C}^n$. To that end, we define a \emph{multi-index} $\alpha=(\alpha_1,\ldots,\alpha_n)$ to be any element of $(\mathbb{N}\cup\{0\})^n$. Then for $z=(z_1,z_2,\ldots,z_n)\in\mathbb{C}^n$ and multi-index $\alpha=(\alpha_1,\alpha_2,\ldots,\alpha_n)$, we define
\[z^\alpha = z_1^{\alpha_1} z_2^{\alpha_2}\cdot\ldots\cdot z_n^{\alpha_n}.\]
Note that $z^\alpha$ is holomorphic. Furthermore, we define $|\alpha|$ and $\alpha!$ as
\begin{equation*}
|\alpha| = \alpha_1+\alpha_2+\ldots+\alpha_n,\, \alpha! = \alpha_1!\alpha_2!\cdot\ldots\cdot\alpha_n!\,. 
\end{equation*}
Finally, for any multi-index $\omega\in(\mathbb{N}\cup\{0\})^n$ we define the constant $c_{\omega}$ as
\begin{equation}
c_{\omega} := \frac{(n-1)!\omega!}{(n-1+|\omega|)!}.
\end{equation}
One reason the constants $c_{\omega}$ are significant is due to the following result (found as Propositions 1.4.8 and 1.4.9 in \cite{Rudin1980}):
\begin{theorem}\label{theorem-integral}
Let $\omega,\upsilon\in (\mathbb{N}\cup\{0\})^n$ be multi-indices. Then
\begin{equation}\label{integralEq1}
\int_{S} \zeta^\omega\overline{\zeta}^\upsilon d\sigma(\zeta)=\left\{\begin{array}{lr} 0 & \mbox{if $\omega\neq\upsilon$}\\
c_{\omega} & \mbox{if $\omega=\upsilon$}\end{array}\right..
\end{equation}
\end{theorem}

With this notation in place, we can state the main result of the paper:

\begin{theorem}\label{theorem-main}
Suppose $1\leq p<\infty$ and $f\in L^p(S,\sigma)$. Then $f\in H^p(S)$ if and only if for any pair of multi-indices $\alpha$ and $\beta$ we have the following: \\
\noindent (a) If $\alpha_j>\beta_j$ for some $1\leq j\leq n$, then we have 
\begin{equation}\label{conjectEq1}
\int_{S} \zeta^\alpha\overline{\zeta}^\beta f(\zeta)d\sigma(\zeta)=0.
\end{equation}
(b) If $\alpha_j\leq \beta_j$ for all $1\leq j\leq n$ and we define $\beta-\alpha=(\beta_1-\alpha_1,\beta_2-\alpha_2,\ldots,\beta_n-\alpha_n)$, then
\begin{equation}\label{conjectEq2}
c_{\beta}^{-1}\int_{S}\zeta^\alpha\overline{\zeta}^{\beta}f(\zeta)d\sigma(\zeta)=
c_{\beta-\alpha}^{-1}\int_{S}\overline{\zeta}^{\beta-\alpha}f(\zeta)d\sigma(\zeta).
\end{equation}
\end{theorem}

We believe that the result of Theorem \ref{theorem-main} is new. Admittedly, Theorem \ref{theorem-main} is not as elegant as Theorem \ref{theorem-2}. One reason for this inelegance is due to the fact that when $n>1$, for $\zeta\in S$ and multi-index $\alpha$ the product $\zeta^\alpha\overline{\zeta}^\alpha=|\zeta_1|^{\alpha_1}\cdot\ldots\cdot|\zeta_n|^{\alpha_n}$ can be strictly less than 1 and thus not a constant (whereas when $n=1$ this product will always equal 1). That said, Theorem \ref{theorem-main} reduces to Theorem \ref{theorem-2} in the case that $n=1$. Indeed, if $\alpha_1>\beta_1$ and $\zeta_1\in U$, then $\zeta_1^{\alpha_1}\overline{\zeta_1}^{\beta_1} = \zeta_1^{\alpha_1-\beta_1}$ and so Equation \eqref{conjectEq1} is equivalent to the condition in Theorem \ref{theorem-2}. In the case that $\alpha_1\leq \beta_1$ and $\zeta_1\in U$, then $\zeta_1^{\alpha_1}\overline{\zeta_1}^{\beta_1}=\overline{\zeta_1}^{\beta_1-\alpha_1}$. Furthermore, as $n-1=0$ in this case, the coefficients $c_{\beta}$ and $c_{\beta-\alpha}$ both equal 1. Thus, Equation \eqref{conjectEq2} holds for any $f\in L^1(U,\sigma)$ (which is why there is not a second case present in Theorem \ref{theorem-2}). To be clear, Equation \eqref{conjectEq2} does not generally hold for every $f\in L^1(S,\sigma)$ when $n>1$. For a counterexample, let $n=2$ and consider $\alpha=\beta=(1,1)$ with the (non-holomorphic) function $f(\zeta_1,\zeta_2)=(\zeta_1\zeta_2)\overline{(\zeta_1\zeta_2)}$. Then by Theorem \ref{theorem-integral} the left-hand side of Equation \eqref{conjectEq2} equals $c_{(1,1)}^{-1}c_{(2,2)}=1/5$, 
whereas the right-hand side of Equation \eqref{conjectEq2} equals $c_{(0,0)}^{-1}c_{(1,1)}=1/6$.

The rest of the paper is organized as follows: in Section \ref{section2} we will rigorously define and state known properties of the Poisson and Cauchy integrals $P[f]$ and $C[f]$, respectively, that were previously mentioned in Theorem \ref{theorem-1} and will play a large role in the proof of Theorem \ref{theorem-main}, in Section \ref{section3} we will prove Theorem \ref{theorem-main}, and in Section \ref{section4} we make some concluding remarks. As many of the background theorems that we will cite come from \cite{Rudin1980}, we have and will continue to mostly adopt its notation to make it easier to look up results while reading this paper.

\section{The Poisson and Cauchy Integrals}\label{section2}

Let $z=(z_1,z_2,\ldots,z_n)$ and $w=(w_1,w_2,\ldots,w_n)$ be elements of $\mathbb{C}^n$. Define the inner product $\langle z,w\rangle$ and length $|z|$ as
\[\langle z,w\rangle =z_1\overline{w_1}+ z_2\overline{w_2}+\ldots+z_n\overline{w_n},\, |z| = \langle z,z\rangle^{1/2}.\]

Next let $z\in B$ and $\zeta\in S$. Define the (invariant) Poisson kernel $P(z,\zeta)$ and Poisson integral $P[f]$ for any $f\in L^1(S,\sigma)$ as
\[P(z,\zeta)=\frac{(1-|z|^2)^n}{|1-\langle z,\zeta\rangle|^{2n}},\, P[f](z)=\int_{S} P(z,\zeta) f(\zeta)d\sigma(\zeta).\]
Let $z,w\in \mathbb{C}^n$ with $\langle z,w\rangle \neq 1$. Define the Cauchy kernel $C(z,w)$ and Cauchy integral $C[f]$ for any $f\in L^1(S,\sigma)$ as
\[C(z,w)=\frac{1}{(1-\langle z, w\rangle)^n},\, C[f](z)=\int_{S} C(z,\zeta) f(\zeta)d\sigma(\zeta) \mbox{ for $z\in B$}.\]

The Poisson and Cauchy integrals have very nice properties. In particular, for any $f\in L^p(S,\sigma)$, $C[f]$ is holomorphic on $B$ and $P[f]$ is $\mathcal{M}$-harmonic. To the latter classification, a function $F$ on $B$ is called $\mathcal{M}$-harmonic if it is twice continuously differentiable on $B$ and $\tilde{\Delta}F=0$, where $\tilde{\Delta}$ a certain second-degree partial differential operator that is called the \emph{invariant Laplacian}. When $n=1$, $\mathcal{M}$-harmonic functions correspond to harmonic functions, but these classes of functions are different for $n>1$. Furthermore, any function that is holomorphic on $B$ is also $\mathcal{M}$-harmonic on $B$. For our purposes we do not need the definition of $\tilde{\Delta}$, but the definition and some fundamental results on $\mathcal{M}$-harmonic functions can be found in Chapter 4 of \cite{Rudin1980}. Since we are looking for holomorphic functions on $B$ you may wonder why we consider the Poisson integral at all. It is because it has nice convergence properties at the boundary $S$ as Theorem \ref{theorem-3} will show. To state this theorem, we need two definitions: If $u:B\to\mathbb{C}$, for any $0\leq r<1$ we define the function $u_r:S\to\mathbb{C}$ as
\[u_r(\zeta) = u(r\zeta).\]
%
\begin{theorem}\label{theorem-3}
Let $1\leq p <\infty$ and let $f\in L^p(S,\sigma)$.
\begin{enumerate}[(a)]
\item For any $0\leq r<1$, we have $\|P[f]_r\|_p\leq \|f\|_p$ and $\lim_{r\to 1^-}\|P[f]_r-f\|_p=0$. (Theorem 3.3.4 part (b) of \cite{Rudin1980})
\item For almost all $\zeta\in S$ (with respect to $\sigma$), the $K$-limit of $P[f](\zeta)$ exists and equals $f(\zeta)$ (more succinctly, $P[f]^*= f$). (Theorem 5.4.8 of \cite{Rudin1980})
\end{enumerate}
\end{theorem}

Thus $P[f]$ converges to $f$ on $S$ in both an $L^p$-sense and in terms of $K$-limits. Note that Theorem \ref{theorem-3} implies that, loosely speaking, \textit{all} functions in $L^p(S,\sigma)$ are boundary traces of a $\mathcal{M}$-harmonic functions on $B$ (that is, $\mathcal{M}$-harmonic are special but not \textit{special}). In contrast, the Cauchy integral $C[f]$ is not as well-behaved; while $C[f]$ does have $K$-limits at almost all points of $S$ (see the corollary to Theorem 6.2.3 in \cite{Rudin1980}), these $K$-limits need not equal the original function $f$. Indeed, for $1<p<\infty$, the mapping $f\mapsto C[f]^*$ is a linear projection from $L^p(S,\sigma)$ onto $H^p(S)$ (see the corollary to Theorem 6.3.1 in \cite{Rudin1980}).

If the condition $P[f]=C[f]$ is satisfied and we let $F$ equal both, then we have the best of both worlds. Indeed, in this case $F$ converges to $f$ on $S$ in terms of $K$-limits (by part (b) of Theorem \ref{theorem-3} since $F=P[f]$), $F$ is holomorphic (since $F=C[f]$), and $F\in H^p(B)$ (by part (a) of Theorem \ref{theorem-3} since $F=P[f]$). By part (a) of Theorem \ref{theorem-1}, we see that $f=F^*$ must lie in $H^p(S)$. That is, we have proven the following lemma:

\begin{lemma}\label{lemma-1}
Let $1\leq p <\infty$. If $f\in L^p(S,\sigma)$ and $P[f]=C[f]$, then $f\in H^p(S)$.
\end{lemma}

Lemma \ref{lemma-1} will be central to our proof of Theorem \ref{theorem-main} which we are now in a position to prove.    

\section{Proof of Theorem \ref{theorem-main}}\label{section3}

\begin{proof}[Proof of Theorem \ref{theorem-main}]
First we prove the forward direction. Suppose $f\in H^p(S)$ and let $F=P[f]$. By part (b) of Theorem \ref{theorem-1}, $F\in H^p(B)$ and by part (a) of Theorem \ref{theorem-3}, we have $\int_{S} |f(\zeta)-F(r\zeta)|^p d\sigma(\zeta)\to 0$ as $r\to 1^-$. Since $F$ is holomorphic on $B$, we can write it as a power series of the form $\sum_{\gamma} c_\gamma z^\gamma$ where the series is taken over all multi-indices $\gamma$ and converges to $F$ uniformly on compact subsets of $B$ (see Remark (i) in Subsection 1.2.6 in \cite{Rudin1980}). Let $\alpha$ and $\beta$ be a pair of multi-indices. First suppose that $\alpha_j>\beta_j$ for some $1\leq j\leq n$. Then for any other multi-index $\gamma$ we have $\alpha_j+\gamma_j>\beta_j$. Thus, $\alpha+\gamma\neq \beta$ and so by Theorem \ref{theorem-integral} we have
\[\int_{S} \zeta^{\alpha+\gamma}\overline{\zeta}^\beta d\sigma(\zeta)=0.\]
Let $0<r<1$. By the uniform convergence of $F(z)=\sum_{\gamma} c_\gamma z^\gamma$ on the sphere of radius $r$, the above implies that 
\[\int_{S} \zeta^{\alpha}\overline{\zeta}^\beta F(r\zeta)d\sigma(\zeta)=0.\]
Thus, using the fact that $|\zeta^{\alpha}\overline{\zeta}^\beta|\leq 1$, we have
\begin{eqnarray*}
\left|\int_{S} \zeta^{\alpha}\overline{\zeta}^\beta f(\zeta)d\sigma(\zeta)\right| &=& \left|\int_{S} \zeta^{\alpha}\overline{\zeta}^\beta(f(\zeta)-F(r\zeta))d\sigma(\zeta)\right|\\
&\leq & \int_{S}|f(\zeta)-F(r\zeta)|d\sigma(\zeta)\leq \left(\int_{S}|f(\zeta)-F(r\zeta)|^pd\sigma(\zeta)\right)^{1/p}.
\end{eqnarray*}
As $r\to 1^-$, the right-hand expression goes to $0$ and so $\int_{S} \zeta^{\alpha}\overline{\zeta}^\beta f(\zeta)d\sigma(\zeta)=0$, proving Equation \eqref{conjectEq1}. Now suppose $\alpha_j\leq\beta_j$ for all $1\leq j\leq n$ and define the multi-index $\lambda=\beta-\alpha$. Note that for any multi-index $\gamma$, again by Theorem \ref{theorem-integral} we have
\[c_{\beta}^{-1}\int_{S}\zeta^\alpha\overline{\zeta}^{\beta}\zeta^\gamma d\sigma(\zeta)=\left\{\begin{array}{lr}0 & \mbox{if $\gamma\neq\lambda$}\\1 & \mbox{if $\gamma=\lambda$}\end{array}\right\}=
c_{\lambda}^{-1}\int_{S}\overline{\zeta}^{\lambda}\zeta^\gamma d\sigma(\zeta).\]
Again by the uniform convergence of $F(z)=\sum_{\gamma} c_\gamma z^\gamma$ on the sphere of radius $r$, the above implies that 
\begin{equation*}
c_{\beta}^{-1}\int_{S}\zeta^\alpha\overline{\zeta}^{\beta}F(r\zeta)d\sigma(\zeta)=
c_{\lambda}^{-1}\int_{S}\overline{\zeta}^{\lambda}F(r\zeta)d\sigma(\zeta).
\end{equation*}
Using an analogous argument as we used in to prove Equation \eqref{conjectEq1}, the above implies that Equation \eqref{conjectEq2}, completing the forward direction.

For the backward direction, suppose $f\in L^p(S,\sigma)$ and satisfies Equations \eqref{conjectEq1} and \eqref{conjectEq2}. We will repeatedly cite theorems that are standard in a graduate level course in real analysis to swap a sum through an integral (these theorems can be found in \cite{greenRudin}, for example). By Lemma \ref{lemma-1}, to show that $f\in H^p(S)$ it suffices to show that $P[f](z)=C[f](z)$ for all $z\in B$. First we will show that for any $z,w\in \overline{B}$ where at least one of $z$ or $w$ lies in $B$ we have
\begin{equation}\label{CauchyKernelSum}
C(z,w) = \frac{1}{(1-\langle z,w\rangle)^n} = \sum_{\omega}c_{\omega}^{-1} z^\omega\overline{w}^\omega,
\end{equation}
where the sum is taken over all multi-indices $\omega$. First note that $|z||w|<1$. Using the power series $1/(1-x)^n=\sum_{j=0}^\infty \begin{pmatrix}j+n-1\\ n-1 \end{pmatrix} x^j$ and the multinomial theorem, we have 
\begin{eqnarray*}
C(z,w) &=& \frac{1}{(1-\langle z,w\rangle)^n} = \sum_{j=0}^\infty \begin{pmatrix}j+n-1\\ n-1 \end{pmatrix}(\langle z, w\rangle)^j
= \sum_{j=0}^\infty \begin{pmatrix}j+n-1\\ n-1 \end{pmatrix} \sum_{|\omega|=j} \begin{pmatrix}j \\ \omega_1, \omega_2,\ldots, \omega_n \end{pmatrix} z^\omega \overline{w}^\omega\\ 
&=& \sum_{\omega} \begin{pmatrix}|\omega|+n-1\\ n-1 \end{pmatrix}\begin{pmatrix}|\omega| \\ \omega_1, \omega_2,\ldots, \omega_n \end{pmatrix} z^\omega \overline{w}^\omega= \sum_{\omega} \frac{(|\omega|+n-1)!}{(n-1)!\omega !} z^\omega \overline{\zeta}^\omega = \sum_{\omega}c_{\omega}^{-1} z^\omega\overline{w}^\omega.
\end{eqnarray*}
The rearranging of the sum above is justified as it absolutely converges. Indeed, using the Cauchy-Schwarz inequality we have
\begin{eqnarray}
\lefteqn{\sum_{j=0}^\infty \begin{pmatrix}j+n-1\\ n-1 \end{pmatrix} \sum_{|\omega|=j} \begin{pmatrix}j \\ \omega_1, \omega_2,\ldots, \omega_n \end{pmatrix} |z^\omega \overline{w}^\omega|}\nonumber\\
&=& \sum_{j=0}^\infty \begin{pmatrix}j+n-1\\ n-1 \end{pmatrix}\left(\sum_{k=1}^n |z_k||w_k|\right)^j  \leq \sum_{j=0}^\infty \begin{pmatrix}j+n-1\\ n-1 \end{pmatrix}(|z||w|)^j= \frac{1}{(1-|z||w|)^n},\label{absoluteConEq1}
\end{eqnarray}
as $|z||w|<1$. Thus, again using absolute convergence, we have for $z\in B$ and $\zeta\in S$
\begin{equation}\label{PoissonKernelSum}
P(z,\zeta)=\frac{C(z,\zeta)C(\zeta,z)}{C(z,z)}=C(z,z)^{-1}\sum_{\upsilon}\sum_{\omega}c_{\omega}^{-1}c_{\upsilon}^{-1}z^\omega\overline{\zeta}^\omega \zeta^\upsilon\overline{z}^\upsilon,
\end{equation}
where each sum is taken over all multi-indices $\omega$ and $\upsilon$. Furthermore, using a similar computation as in Inequality \eqref {absoluteConEq1}, and using the assumption that $f\in L^p(S,\sigma)\subseteq L^1(S,\sigma)$ we have
\begin{equation*}
\int_{S} \sum_{\upsilon}\sum_{\omega}c_{\omega}^{-1}c_{\upsilon}^{-1}\left|z^\omega\overline{\zeta}^\omega \zeta^\upsilon\overline{z}^\upsilon f(\zeta)\right|d\sigma(\zeta)\leq \int_{S} \frac{1}{(1-|z||\zeta|)^{2n}}|f(\zeta)|d\sigma(\zeta)\leq \frac{1}{(1-|z|)^{2n}} \int_{S}|f(\zeta)|d\sigma(\zeta)  <\infty.
\end{equation*}
Thus, by Equation \eqref{PoissonKernelSum} and Fubini's Theorem we have
\begin{equation}\label{conjectEq3}
P[f](z) = \int_{S} P(z,\zeta)f(\zeta)d\sigma(\zeta) = C(z,z)^{-1}\sum_{\upsilon}\sum_{\omega}c_{\omega}^{-1}c_{\upsilon}^{-1}z^\omega\overline{z}^\upsilon\int_{S}\zeta^\upsilon\overline{\zeta}^\omega  f(\zeta)d\sigma(\zeta).
\end{equation}
By Equation \eqref{conjectEq1}, the integrals of $\int_{S}\zeta^\upsilon\overline{\zeta}^\omega  f(\zeta)d\sigma(\zeta)$ vanish except when $\upsilon_j\leq \omega_j$ for each $1\leq j\leq k$. That is, the integrals vanish except when $\omega=\upsilon+\lambda$ for some multi-index $\lambda$. Thus, Equation \eqref{conjectEq3} becomes
\begin{eqnarray*}
P[f](z) 
&=& C(z,z)^{-1}\sum_{\upsilon} \sum_{\lambda}c_{\upsilon}^{-1}z^{\upsilon+\lambda}\overline{z}^\upsilon \left(c_{\upsilon+\lambda}^{-1}\int_{S}\zeta^\upsilon\overline{\zeta}^{\upsilon+\lambda}  f(\zeta)d\sigma(\zeta)\right).
\end{eqnarray*}
Using Equation \eqref{conjectEq2} in the summation above yields
\begin{eqnarray*}
P[f](z) &=& C(z,z)^{-1}\sum_{\upsilon}\sum_{\lambda}c_{\upsilon}^{-1}z^{\upsilon+\lambda}\overline{z}^\upsilon \left(c_{\lambda}^{-1}\int_{S}\overline{\zeta}^{\lambda}  f(\zeta)d\sigma(\zeta)\right)\\
&=& C(z,z)^{-1}\left(\sum_{\upsilon}c_{\upsilon}^{-1}z^{\upsilon}\overline{z}^\upsilon\right)\left(\int_{S}\sum_{\lambda}c_{\lambda}^{-1}z^{\lambda} \overline{\zeta}^{\lambda} f(\zeta)d\sigma(\zeta) \right)\\
&=& C(z,z)^{-1}C(z,z)\int_{S} C(z,\zeta)f(\zeta)d\sigma(\zeta) = C[f](z),
\end{eqnarray*}
where we can use the Dominated Convergence Theorem with dominating function $\frac{1}{(1-|z||\zeta|)^n}|f(\zeta)|$ to justify swapping the sum and the integral in the second to last line, and used Equation \eqref{CauchyKernelSum} in the last line, completing the proof.
\end{proof}
\section{Closing Remarks: Can Theorem \ref{theorem-main} Be Generalized?}\label{section4}
In the case of one complex variable, Hardy spaces can be defined on a variety of different subsets of $\mathbb{C}$. In particular, in the case that $\Omega\subseteq \mathbb{C}$ is a simply-connected open set with some minimal smoothness assumptions on its boundary $\partial\Omega$, then the Hardy spaces $H^p(\Omega)$ and $H^p(\partial \Omega)$ have several equivalent definitions (see \cite{Lanzani2000}). In this general setting, the statement of Theorem \ref{theorem-2} still holds with $\partial\Omega$ replacing $U$ (see Theorem 10.4 in \cite{Duren1970}). So while one might hope that Theorem \ref{theorem-main} could carry over to Hardy spaces of a wide class of open subsets of $\mathbb{C}^n$, there are potential difficulties in proving it. For example, the proof of the generalization of Theorem \ref{theorem-2} in \cite{Duren1970} uses the Riemann Mapping Theorem and the Riemann Mapping Theorem does not hold in $\mathbb{C}^n$ if $n>1$ (see the Corollary to Theorem 2.2.4 in \cite{Rudin1980}), removing one possible strategy for proving a generalization of Theorem \ref{theorem-main}. Also, the technical machinery needed for the Hardy space setting in $\mathbb{C}^n$ for $n>1$ gets pretty gnarly pretty quickly (see \cite{LanzaniStein2015}). So even if Theorem \ref{theorem-main} is true in a more general setting, it will likely not a short proof like we have presented here. In any case, the main point of this note is that boundary traces of holomorphic functions form a very \textit{special} club, a fact that is almost certainly true for more exotic domains in $\mathbb{C}^n$ than just the ball $B$.\\

\noindent\textbf{Acknowledgement}: I thank the American Institute of Mathematics where this work was initiated. I also thank my colleague Linda McGuire for giving me advice on an early draft of this manuscript, and my colleagues Loredana Lanzani and Yuan Zhang for our many fruitful conversations on the wonderful world of Hardy spaces.

\end{document}